\documentclass{aims}

\title[Dual Pairs and Regularization of Kummer Shapes]{Dual Pairs and Regularization\\of Kummer Shapes in Resonances}
\author[Tomoki Ohsawa]{}
\email{tomoki@utdallas.edu}
\date{\today}

\dedicatory{In honor of Darryl Holm's 70th birthday}



\usepackage{graphicx}
\usepackage{tikz-cd}

\usepackage[colorlinks=true]{hyperref}
\hypersetup{urlcolor=blue, citecolor=red}

\def\currentvolume{X}
 \def\currentissue{X}
  \def\currentyear{200X}
   \def\currentmonth{XX}
    \def\ppages{X--XX}
     \def\DOI{10.3934/xx.xx.xx.xx}

\theoremstyle{plain}
\newtheorem{theorem}{Theorem}[section]
\newtheorem{corollary}[theorem]{Corollary}
\newtheorem{lemma}[theorem]{Lemma}
\newtheorem{proposition}[theorem]{Proposition}

\theoremstyle{definition}

\newtheorem{example}[theorem]{Example}

\theoremstyle{remark}
\newtheorem{remark}[theorem]{Remark}

\def\od#1#2{\dfrac{d#1}{d#2}}
\def\pd#1#2{\dfrac{\partial #1}{\partial #2}}
\def\tpd#1#2{\partial #1/\partial #2}

\def\parentheses#1{\!\left(#1\right)}

\def\braces#1{\!\left\{#1\right\}}

\def\tr{\mathop{\mathrm{tr}}\nolimits}
\def\Span{\operatorname{span}} 

\def\diag{\operatorname{diag}}

\def\norm#1{\left\|#1\right\|}

\def\R{\mathbb{R}}
\def\C{\mathbb{C}}

\def\N{\mathbb{N}}
\def\defeq{\mathrel{\mathop:}=}
\def\eqdef{=\mathrel{\mathop:}}
\def\setdef#1#2{ \left\{ #1 \ |\ #2 \right\} }
\def\ip#1#2{\left\langle#1,#2\right\rangle}

\renewcommand{\Re}{\operatorname{Re}}
\renewcommand{\Im}{\operatorname{Im}}
\def\eps{\varepsilon}
\def\rmi{{\rm i}}

\def\d{\mathbf{d}}
\def\ins#1{{\bf i}_{#1}}
\def\PB#1#2{\left\{#1,#2\right\}}
\newcommand\Ad{\operatorname{Ad}}
\newcommand\ad{\operatorname{ad}}

\def\SO{\mathsf{SO}}

\def\U{\mathsf{U}}
\def\SU{\mathsf{SU}}

\def\u{\mathfrak{u}}

\def\su{\mathfrak{su}}

\newenvironment{tbmatrix}{\left[\begin{smallmatrix}}{\end{smallmatrix}\right]}

\begin{document}
\maketitle
\centerline{\scshape Tomoki Ohsawa$^*$}
\medskip
{\footnotesize
  \centerline{Department of Mathematical Sciences}
  \centerline{The University of Texas at Dallas}
  \centerline{800 W Campbell Rd}
  \centerline{Richardson, TX 75080-3021, USA}
} 

\bigskip

 \centerline{(Communicated by the associate editor name)}

\begin{abstract}
  We present an account of dual pairs and the Kummer shapes for $n:m$ resonances that provides an alternative to Holm and Vizman's work.
  The advantages of our point of view are that the associated Poisson structure on $\mathfrak{su}(2)^{*}$ is the standard $(+)$-Lie--Poisson bracket independent of the values of $(n,m)$ as well as that the Kummer shape is regularized to become a sphere without any pinches regardless of the values of $(n,m)$.
  A similar result holds for $n:-m$ resonance with a paraboloid and $\mathfrak{su}(1,1)^{*}$.
  The result also has a straightforward generalization to multidimensional resonances as well.
\end{abstract}

\section{Introduction}
\subsection{Kummer Shapes and Dual Pairs in Resonances}
Hamiltonian systems with resonant symmetry have been studied quite extensively from many different perspectives.
Resonant symmetry crops up in many different forms of $\mathbb{S}^{1}$ symmetries.
Although it is one of the simplest symmetries geometrically, it is not only rich in examples and applications but also possesses interesting mathematical structures; see, e.g., Holm~\cite[Chapters~4--6]{Ho2011a}, Dullin~et~al.~\cite{DuGiCu2004}, Haller~\cite[Chapter~4]{Haller1999}, and references therein.

From the geometric point of view, Churchill et~al.~\cite{ChKuRo1983}, Kummer~\cite{Ku1981,Ku1986,Ku1987} made a seminal contribution by introducing what is now often referred to as the \textit{Kummer shapes}.
Recently Holm and Vizman~\cite{HoVi2012} discovered a Poisson-geometric structure behind the Kummer shapes by finding a \textit{dual pair} of Poisson maps (see, e.g., Weinstein~\cite{We1983} and Ortega and Ratiu~\cite[Chapter~11]{OrRa2004}) in $n:m$ resonances.

\subsection{Main Results and Outline}
We build on the work of Holm~\cite[Chapter~4]{Ho2011a} and Holm and Vizman~\cite{HoVi2012} to provide an alternative view of the dual pair constructed in \cite{HoVi2012} as well as of the Kummer shapes in $n:m$, $n:-m$, and multidimensional resonances.

Our approach is to relate $n:m$ resonance with any $(n,m) \in \N^{2}$ with the $1:1$ resonance case; this relationship along with the dual pair from \cite{HoVi2012} (see also Golubitsky et~al.~\cite{GoStMa1987}) for $1:1$ resonance naturally gives rise to the dual pair for $n:m$ resonance; see Theorem~\ref{thm:dual_pair-n:m}.
Our dual pair for $n:m$ resonances is slightly different from that of \cite{HoVi2012}.
Specifically, the Poisson structure on $\su(2)^{*}$ in our dual pair is the standard $(+)$-Lie--Poisson structure regardless of the values of $(n,m) \in \N^{2}$.
This is in contrast to the Poisson structure in \cite{HoVi2012} that depends on the values of $(n,m) \in \N^{2}$.
An advantage of this result is that the reduced dynamics in $\su(2)^{*}$ becomes a standard Lie--Poisson dynamics.

A byproduct of this construction is that the Kummer shapes---which usually arise as various shapes such as beet, lemon, onion, turnip, etc.~depending on the values of $n$ and $m$~\cite[Section~4.4.2]{Ho2011a}---are all ``regularized'' to become a sphere.

Section~\ref{sec:n:-m} shows that a similar approach works between $n:-m$ resonance and $1:-1$ resonance.
In this case, again all the Kummer shapes are regularized to become a paraboloid.

We also show, in Section~\ref{sec:MultiD_Resonance}, that the argument for $n:m$ resonances easily generalizes to multi-dimensional resonances.

\section{Kummer Shapes and Dual Pairs in $n:m$ Resonances}
We first briefly review Hamiltonian dynamics with $n:m$ resonant symmetry following Holm~\cite[Chapter~4]{Ho2011a} and Holm and Vizman~\cite{HoVi2012}.
We then find a Poisson map that provides a bridge between $n:m$ resonances and the $1:1$ resonance using a change of variables introduced in \cite[Section~A.5.4]{Ho2011a}.
This Poisson map naturally gives rise to a dual pair of Poisson maps for $n:m$ resonances with the standard $(+)$-Lie--Poisson bracket on $\su(2)^{*}$ by relating it to the dual pair for $1:1$ resonance from Golubitsky et~al.~\cite{GoStMa1987} and Holm and Vizman~\cite{HoVi2012}.
This gives an alternative account of the dual pairs in $n:m$ resonances that is slightly different from those in Holm and Vizman~\cite{HoVi2012}.
In fact, the Kummer shapes~\cite{Ku1981,ChKuRo1983,Ku1986,Ku1987} turn out to be spheres regardless of the values of $n$ and $m$.
We work out an example to illustrate this result, as well as extend the result to $n:-m$ resonances.

\subsection{{\boldmath $n:m$} Resonances}
Let $\mathbb{S}^{1} = \setdef{ e^{\rmi\theta} \in \C }{ \theta \in [0, 2\pi) }$ and $\C_{\times} \defeq \C\backslash\{0\}$ be the set of non-zero complex numbers, and set
\begin{equation*}
  \C_{\times}^{2} = \setdef{ \mathbf{a} = (a_{1},a_{2}) }{ a_{1}, a_{2} \in \C_{\times} }.
\end{equation*}
We equip the manifold $\C_{\times}^{2}$ with the symplectic form
\begin{equation}
  \label{eq:Omega-C2}
  \Omega_{\C_{\times}^{2}} \defeq -\frac{\rmi}{2} \sum_{j=1}^{2} \d a_{j} \wedge \d\bar{a}_{j} = -\d\Theta_{\C_{\times}^{2}},
\end{equation}
where
\begin{equation*}
  \Theta_{\C_{\times}^{2}} \defeq \frac{1}{2} \sum_{j=1}^{2} \Im(\bar{a}_{j}\d a_{j}).
\end{equation*}
The associated Poisson bracket is
\begin{equation*}
  \PB{F}{G}_{\C_{\times}^{2}} \defeq 2\rmi \sum_{j=1}^{2} \parentheses{
    \pd{F}{a_{j}} \pd{G}{\bar{a}_{j}} - \pd{G}{a_{j}} \pd{F}{\bar{a}_{j}}
  }.
\end{equation*}

Let $n, m \in \N$ be a pair of natural numbers and consider the following $\mathbb{S}^{1}$-action on $\C_{\times}^{2}$:
\begin{equation}
  \label{eq:Psi}
  \Psi^{n:m}_{(\cdot)}\colon \mathbb{S}^{1} \times \C_{\times}^{2} \to \C_{\times}^{2};
  \qquad
  (e^{\rmi\theta},(a_{1},a_{2})) \mapsto (e^{\rmi n\theta}a_{1}, e^{\rmi m\theta}a_{2}) \eqdef \Psi^{n:m}_{\theta}(\mathbf{a}).
\end{equation}
The corresponding infinitesimal generator is defined for any $\omega \in T_{1}\mathbb{S}^{1} \cong \R$ as follows:
\begin{equation*}
  \omega^{n:m}_{\C_{\times}^{2}}(\mathbf{a})
  = \left.\od{}{\eps} \Psi^{n:m}_{\eps\omega}(\mathbf{a}) \right|_{\eps=0}
  = \rmi\,\omega \parentheses{ n a_{1} \pd{}{a_{1}} + m a_{2} \pd{}{a_{2}} } + \text{c.c.},
\end{equation*}
where ``$\text{c.c.}$'' stands for the complex conjugate of the preceding terms.
This is essentially equivalent to the dynamics of two harmonic oscillators with frequencies $n$ and $m$:
\begin{equation}
  \label{eq:nm_resonant_dynamics}
  \dot{a}_{1} = \rmi n a_{1},
  \qquad
  \dot{a}_{2} = \rmi m a_{2}.
\end{equation}
One also sees that this is the Hamiltonian vector field corresponding to the function $(n |a_{1}|^{2} + m |a_{2}|^{2})/2$.

\subsection{{\boldmath$n:m$} Resonance vs.~{\boldmath$1:1$} Resonance}
Consider the map
\begin{equation}
  \label{eq:f_n:m}
  f_{n:m}\colon \C_{\times}^{2} \to \C_{\times}^{2};
  \qquad
  (a_{1},a_{2}) \mapsto \parentheses{ \frac{a_{1}^{m}}{\sqrt{m}\,|a_{1}|^{m-1}}, \frac{a_{2}^{n}}{\sqrt{n}\,|a_{2}|^{n-1}} }.
\end{equation}
This change or coordinates is briefly mentioned in Holm~\cite[Section~A.5.4]{Ho2011a}, and is a slight modification of the change of variables introduced in \cite[Section~4.4]{Ho2011a}, where $\sqrt{m}$ and $\sqrt{n}$ are $m$ and $n$ respectively instead.
Note that the map is not one-to-one and hence is not invertible in general.

Let $\mathbf{b} = (b_{1},b_{2})$ be the coordinates for the second copy of $\C_{\times}^{2}$, and equip $\C_{\times}^{2} = \{ (b_{1},b_{2}) \}$ with the same symplectic structure $\Omega_{\C_{\times}^{2}}$ defined in \eqref{eq:Omega-C2} above, and hence with the same Poisson bracket as the above, i.e.,
\begin{equation}
  \label{eq:PB-b}
  \PB{F}{G}_{\C_{\times}^{2}} \defeq 2\rmi \sum_{j=1}^{2}\parentheses{ \pd{F}{b_{j}} \pd{G}{\bar{b}_{j}} - \pd{G}{b_{j}} \pd{F}{\bar{b}_{j}} }.
\end{equation}
Then it is straightforward calculations (see the proof of Proposition~\ref{prop:f_n} below) to see that $f_{n:m}$ is a Poisson map, i.e.,
\begin{equation*}
  \PB{F\circ f_{n:m}}{G\circ f_{n:m}}_{\C_{\times}^{2}} = \PB{F}{G}_{\C_{\times}^{2}}\circ f_{n:m}.
\end{equation*}
One also sees that $f_{n:m}$ is a local symplectomorphism with respect to $\Omega_{\C_{\times}^{2}}$ as well, i.e., for any $\mathbf{a} \in \C_{\times}^{2}$, there exists an open neighborhood $U$ of $\mathbf{a}$ in $\C_{\times}^{2}$ such that $f_{n:m}|_{U}\colon U \to f_{n:m}(U)$ is symplectic.
In fact, $f_{n:m}$ is a local diffeomorphism because those distinct points $a_{1}, \tilde{a}_{1} \in \C_{\times}$ such that $a_{1}^{m}/(\sqrt{m}\,|a_{1}|^{m-1}) = \tilde{a}_{1}^{m}/(\sqrt{m}\,|\tilde{a}_{1}|^{m-1})$ are on the same circle (i.e., $|a_{1}| = |\tilde{a}_{1}|$) but are separated by angles $2k\pi/m$ with $k = 1, \dots, m-1$; the same goes with the second portion of $f_{n:m}$.
The (local) symplecticity follows from similar coordinate calculations as above; again see the proof of Proposition~\ref{prop:f_n} below for more details.

Let us also define $R_{n:m}\colon \C_{\times}^{2} \to \R$ by
\begin{equation*}
  R_{n:m}(\mathbf{a}) \defeq \frac{1}{2}\parentheses{ \frac{|a_{1}|^{2}}{m} + \frac{|a_{2}|^{2}}{n} }.
\end{equation*}
Clearly it satisfies $R_{n:m} = R_{1:1} \circ f_{n:m}$, and $n m R_{n:m}$ is the Hamiltonian function whose corresponding vector field gives \eqref{eq:nm_resonant_dynamics}, i.e., $R_{n:m}$ is essentially the momentum map corresponding to the action \eqref{eq:Psi}.

Now consider the following natural action of the special unitary group $\SU(2)$ on $\C_{\times}^{2}$:
\begin{equation}
  \label{eq:Phi-SU2}
  \Phi_{(\cdot)}\colon \SU(2) \times \C_{\times}^{2} \to \C_{\times}^{2};
  \qquad
  (U,\mathbf{b}) \mapsto U\mathbf{b} \eqdef \Phi_{U}(\mathbf{b}).
\end{equation}
It is then clear that $R_{1:1}$ is invariant under the action, i.e., $R_{1:1} \circ \Phi_{U} = R_{1:1}$ for any $U \in \SU(2)$.
The momentum map $\mathbf{J}_{1:1}\colon \C_{\times}^{2} \to \su(2)^{*}$ corresponding to the above action is then given by
\begin{align}
  \mathbf{J}_{1:1}(\mathbf{b})
  &= \rmi\parentheses{ \mathbf{b}\mathbf{b}^{*} - \frac{1}{2}|\mathbf{b}|^{2} I } \nonumber\\
  &= \rmi
    \begin{bmatrix}
      \frac{1}{2}(|b_{1}|^{2} - |b_{2}|^{2}) & b_{1}\bar{b}_{2} \smallskip\\
      b_{2}\bar{b}_{1} & -\frac{1}{2}(|b_{1}|^{2} - |b_{2}|^{2})
    \end{bmatrix} \nonumber\\
  &= \parentheses{ \Re(b_{1}\bar{b}_{2}),\, \Im(b_{1}\bar{b}_{2}),\, \frac{1}{2}(|b_{1}|^{2} - |b_{2}|^{2}) }.
    \label{eq:J_1:1}
\end{align}
See Lemma~\ref{lem:J_1} below for a generalization of this result and a proof.
Note that we also identified $\su(2) \cong \su(2)^{*}$ with $\R^{3}$ as follows:
\begin{equation*}
  \su(2)^{*} \cong \su(2) \to \R^{3};
  \qquad
  \rmi
  \begin{bmatrix}
    \xi_{3} & \xi_{1} + \rmi \xi_{2} \\
    \xi_{1} - \rmi \xi_{2} & -\xi_{3}
  \end{bmatrix}
  \mapsto \boldsymbol{\xi} = (\xi_{1}, \xi_{2}, \xi_{3}) .
\end{equation*}
Clearly $\mathbf{J}_{1:1}$ is equivariant, i.e., $\mathbf{J}_{1:1} \circ \Phi_{U} = \Ad_{U^{-1}}^{*}\mathbf{J}_{1:1}(\mathbf{b})$ for any $U \in \SU(2)$.

\subsection{The Lie--Poisson Bracket}
Let $\su(2)^{*}_{+}$ be $\su(2)^{*}$ equipped with $(+)$-Lie--Poisson bracket: For any $F, G \in C^{\infty}(\su(2)^{*})$,
\begin{equation}
  \label{eq:LPB-su2}
  \PB{F}{G}_{+}(\mu) \defeq \ip{\mu}{\left[\pd{F}{\mu}, \pd{G}{\mu}\right]}
  = 2 \boldsymbol{\mu} \cdot (\nabla F(\boldsymbol{\mu}) \times \nabla G(\boldsymbol{\mu})),
\end{equation}
where we identified $\su(2)^{*}$ with $\su(2)$ via the inner product
\begin{equation*}
  \ip{\xi}{\eta} \defeq \frac{1}{2}\tr(\xi^{*}\eta) = \boldsymbol{\xi} \cdot \boldsymbol{\eta}
\end{equation*}
on $\su(2)$ and hence $\su(2)^{*}$ is identified with $\R^{3}$ using the identification $\su(2) \cong \R^{3}$ above.
Hence
\begin{equation*}
  \mu = \begin{bmatrix}
    \mu_{3} & \mu_{1} + \rmi \mu_{2} \\
    \mu_{1} - \rmi \mu_{2} & -\mu_{3}
  \end{bmatrix} \in \su(2)^{*},
\end{equation*}
whereas $\boldsymbol{\mu} =  (\mu_{1}, \mu_{2}, \mu_{3}) \in \R^{3}$ here.
Note that the above Poisson bracket satisfy
\begin{equation*}
  \PB{\mu_{i}}{\mu_{j}}_{+} = 2\mu_{k}
\end{equation*}
for any even permutation $(i,j,k)$ of $(1,2,3)$.

Since the $\SU(2)$-action $\Phi$ defined in \eqref{eq:Phi-SU2} is a left action and the momentum map $\mathbf{J}_{1:1}\colon \C_{\times}^{2} \to \su(2)^{*}$ is equivariant, $\mathbf{J}_{1:1}$ is a Poisson map (see, e.g., Marsden and Ratiu~\cite[Theorem~12.4.1]{MaRa1999}) with respect to the Poisson bracket \eqref{eq:PB-b} and \eqref{eq:LPB-su2}, i.e., for any $F, G \in C^{\infty}(\su(2)^{*})$,
\begin{equation*}
  \PB{F}{G}_{+} \circ \mathbf{J}_{1:1} = \PB{F \circ \mathbf{J}_{1:1}}{G \circ \mathbf{J}_{1:1}}_{\C_{\times}^{2}}.
\end{equation*}
In fact, Holm and Vizman~\cite{HoVi2012} (see also \cite{GoStMa1987}) showed that $R_{1:1}$ and $\mathbf{J}_{1:1}$ form a dual pair of Poisson maps:
\begin{equation}
  \label{eq:dual_pair-1:1}
  \begin{tikzcd}[column sep=7ex, row sep=7ex]
    \R & (\C_{\times}^{2}, \Omega_{\C_{\times}^{2}}) \arrow[swap]{l}{R_{1:1}} \arrow{r}{\mathbf{J}_{1:1}} & \su(2)^{*}_{+},
  \end{tikzcd}
\end{equation}
that is, $(\ker T_{\mathbf{a}}R_{1:1})^{\Omega} = \ker T_{\mathbf{a}}\mathbf{J}_{1:1}$ for any $\mathbf{a} \in \C_{\times}^{2}$.

\subsection{{\boldmath $n:m$} Resonance Invariants}
Let us combine the map $f_{n:m}$ from \eqref{eq:f_n:m} and the momentum map $\mathbf{J}_{1:1}$ from \eqref{eq:J_1:1} to define
\begin{equation*}
  \mathbf{J}_{n:m}\colon \C_{\times}^{2} \to \su(2)^{*} \cong \R^{3};
  \qquad
  \mathbf{J}_{n:m} \defeq \mathbf{J}_{1:1} \circ f_{n:m}.
\end{equation*}
In coordinates, we have
\begin{multline*}
  \mathbf{J}_{n:m}(\mathbf{a}) = \Bigl(
    \Re\parentheses{ \frac{a_{1}^{m} \bar{a}_{2}^{n}}{\sqrt{n m} |a_{1}|^{m-1} |a_{2}|^{n-1}} },\,
    \Im\parentheses{ \frac{a_{1}^{m} \bar{a}_{2}^{n}}{\sqrt{n m} |a_{1}|^{m-1} |a_{2}|^{n-1}} },\, \\
    \frac{1}{2}\parentheses{ \frac{|a_{1}|^{2}}{m} - \frac{|a_{2}|^{2}}{n} }
  \Bigr).
\end{multline*}
These are essentially the ``invariants'' (of \eqref{eq:nm_resonant_dynamics} but not necessarily invariants of a general Hamiltonian system in $n:m$ resonance) from \cite[Proposition~4.4.1 on p.~266]{Ho2011a} although the expressions are slightly different.

Note that $\mathbf{J}_{n:m}$ is also slightly different from the corresponding map $\Pi$ in Holm and Vizman~\cite{HoVi2012} as well.
This difference leads to an alternative construction of a dual map as well as different Kummer shapes as we shall see in the next subsection.

\subsection{Dual Pairs and Kummer Shapes}
We are now ready to describe our account of dual pairs and Kummer shapes in $n:m$ resonances.
Specifically, our result identifies a relationship between the dual pair~\eqref{eq:dual_pair-1:1} of the $1:1$ resonance and $n:m$ resonances as well as the momentum map origin of the dual pairs of Poisson maps for $n:m$ resonances.
\begin{theorem}
  \label{thm:dual_pair-n:m}
  The Poisson maps $R_{n:m}\colon \C_{\times}^{2} \to \R$ and $\mathbf{J}_{n:m}\colon \C_{\times}^{2} \to \su(2)^{*}_{+}$ are a dual pair for any pair of natural numbers $(n,m) \in \N^{2}$, i.e., for any $\mathbf{a} \in \C_{\times}^{2}$, $\ker T_{\mathbf{a}}R_{n:m}$ and $\ker T_{\mathbf{a}}\mathbf{J}_{n:m}$ are symplectic orthogonal complements to each other.
  Moreover, the dual pair of Poisson maps for $n:m$ resonances is related to the dual pair of momentum maps $R_{1:1}$ and $\mathbf{J}_{1:1}$ as is shown in the diagram below.
  \begin{equation*}
    \begin{tikzcd}[column sep=9ex, row sep=9ex]
      & (\C_{\times}^{2}, \Omega_{\C_{\times}^{2}}) \arrow{d}{f_{n:m}} \arrow[swap]{dl}{R_{n:m}} \arrow{dr}{\mathbf{J}_{n:m}}& \\
      \R & (\C_{\times}^{2}, \Omega_{\C_{\times}^{2}}) \arrow{l}{R_{1:1}} \arrow[swap]{r}{\mathbf{J}_{1:1}} & \su(2)^{*}_{+}
    \end{tikzcd}
  \end{equation*}
\end{theorem}
\begin{proof}
  We know from Holm and Vizman~\cite[Theorem~3.1]{HoVi2012} that the bottom part constitutes a dual pair: For any $\mathbf{b} \in \C_{\times}^{2}$, $\ker T_{\mathbf{b}}R_{1:1}$ and $\ker T_{\mathbf{b}}\mathbf{J}_{1:1}$ are symplectic orthogonal complements to each other with respect to $\Omega_{\C_{\times}^{2}}$, i.e., $(\ker T_{\mathbf{b}}R_{1:1})^{\Omega} = \ker T_{\mathbf{b}}\mathbf{J}_{1:1}$.
  However, since $R_{n:m} = R_{1:1} \circ f_{n:m}$, we see that, for any $\mathbf{a} \in \C_{\times}^{2}$,
  \begin{equation*}
    T_{\mathbf{a}}R_{n:m} =  T_{f_{n:m}(\mathbf{a})}R_{1:1} \circ T_{\mathbf{a}}f_{n:m}.
  \end{equation*}
  Now recall that $f_{n:m}$ is a local diffeomorphism; so we have
  \begin{equation*}
    \ker T_{\mathbf{a}}R = (T_{\mathbf{a}}f_{n:m})^{-1}(\ker T_{f_{n:m}(\mathbf{a})}R_{1:1}).
  \end{equation*}
  Similarly,
  \begin{equation*}
    \ker T_{\mathbf{a}}\mathbf{J}_{n:m} = (T_{\mathbf{a}}f_{n:m})^{-1}(\ker T_{f_{n:m}(\mathbf{a})}\mathbf{J}_{1:1})
  \end{equation*}
  because $\mathbf{J}_{n:m} = \mathbf{J}_{1:1} \circ f_{n:m}$.
  Since $f_{n:m}$ is a local symplectomorphism with respect to $\Omega_{\C_{\times}^{2}}$, we conclude that $(\ker T_{\mathbf{a}}R_{n:m})^{\Omega} = \ker T_{\mathbf{a}}\mathbf{J}_{n:m}$ for any $\mathbf{a} \in \C_{\times}^{2}$.
\end{proof}

Basic results on dual pairs (see  Weinstein~\cite{We1983} and Ortega and Ratiu~\cite[Chapter~11]{OrRa2004}) imply that the image $\mathbf{J}_{n:m}(R_{n:m}^{-1}(r))$ of the level set $R_{n:m}^{-1}(r)$ of $R_{n:m}$ at any $r > 0$ under the map $\mathbf{J}_{n:m}$ is a symplectic leaf in the image of $\mathbf{J}_{n:m}$ in $\su(2)^{*}$.
This is what Holm~\cite[Section~4.4]{Ho2011a} refers to as an \textit{orbit manifold} or \textit{Kummer shape}.

What does the Kummer shape look like in this setting?
It is well known that $\SU(2)$ is a double cover of $\SO(3)$ and the coadjoint action of $\SU(2)$ in $\su(2)^{*} \cong \R^{3}$ is written as rotations in $\R^{3}$ by corresponding elements in $\SO(3)$, and hence the coadjoint orbit in $\su(2)^{*} \cong \R^{3}$ are spheres; these are the symplectic leaves in $\su(2)^{*}$ or the Kummer shape here.
In fact, setting $\mu = \mathbf{J}_{n:m}(\mathbf{a})$, we see that
\begin{equation*}
  \mu_{1}^{2} + \mu_{2}^{2} + \mu_{3}^{2}
  = R_{n:m}(\mathbf{a})^{2}.
\end{equation*}
Therefore, for any pair $(n,m) \in \N^{2}$, the Kummer shape $\mathbf{J}_{n:m}(R_{n:m}^{-1}(r))$ is a sphere without the north and south poles (which correspond to those cases with $a_{2} = 0$ and $a_{1} = 0$ respectively that were removed from the outset).
To summarize:
\begin{corollary}[Regularization of Kummer shape]
  The Kummer shape formed in $\su(2)^{*}$ using the dual pair from Theorem~\ref{thm:dual_pair-n} is the sphere with radius $R_{n:m}(\mathbf{a})$ centered at the origin with the north and south poles removed for any $(n,m) \in \N^{2}$.
\end{corollary}

\begin{remark}
  This result is seemingly contradictory to those from \cite[Section~4.4.2]{Ho2011a} and \cite{HoVi2012} that the Kummer shapes take all kinds of different pinched spheres such as beet, lemon, onion, turnip, etc.~depending on the values of $n$ and $m$.
  The reason for this apparent contradiction is that our definition of the Poisson map $\mathbf{J}_{n:m}$ is slightly different from theirs, and the map regularizes or un-pinches these various Kummer shapes in their setting to spheres.
\end{remark}

As stated above, an advantage of our setting is that the Poisson structure in $\su(2)^{*}$ is simple and standard---the $(+)$-Lie--Poisson structure on $\su(2)^{*}$---as well as independent of $n$ and $m$, whereas the Poisson structure from \cite{Ho2011a,HoVi2012} is more complicated and dependent on the values of $n$ and $m$.
As a result, a Hamiltonian dynamics in $\C_{\times}^{2}$ with $n:m$ resonant symmetry is reduced to the Lie--Poisson equation
\begin{equation*}
  \dot{\mu} = \PB{\mu}{h}_{+}(\mu)
\end{equation*}
or in the vector form,
\begin{equation}
  \label{eq:LPB-su2_coorinates}
  \dot{\boldsymbol{\mu}} = - 2 \boldsymbol{\mu} \times \nabla h(\boldsymbol{\mu}).
\end{equation}
in $\su(2)^{*}$, where $h\colon \su(2)^{*} \to \R$ is the reduced Hamiltonian defined as $h \circ \mathbf{J}_{n:m} = H$.
The Kummer shape is an invariant submanifold of the dynamics.
More specifically, the Kummer shape as the coadjoint orbit in $\su(2)^{*}$ and regard the above Lie--Poisson system as a Hamiltonian system with respect to the Kirillov--Kostant--Souriau structure (see, e.g., Kirillov~\cite[Chapter~1]{Ki2004} and Marsden and Ratiu~\cite[Chapter~14]{MaRa1999} and references therein) on $\su(2)^{*}$.

The disadvantage of our approach is that the expression for the Hamiltonian $h$ tends to get complicated because of the expression for $\mathbf{J}_{n:m}$.
So it is a trade-off between the simplicities of the reduced Hamiltonian $h$ and the Poisson bracket in $\su(2)^{*}$.

\begin{example}[1:2 resonance]
  We consider the dynamics in $\C_{\times}^{2}$ with respect to the symplectic structure~\eqref{eq:Omega-C2} and the Hamiltonian
  \begin{equation*}
    H(\mathbf{a}) = \Re(a_{1}^{2}\bar{a}_{2}).
  \end{equation*}
  The Hamiltonian system $\ins{X_{H}}\Omega_{\C_{\times}^{2}} = \d{H}$ yields
  \begin{equation*}
    \dot{a}_{1} = 2\rmi\,\bar{a}_{1} a_{2},
    \qquad
    \dot{a}_{2} = \rmi\, a_{1}^{2}.
  \end{equation*}
  Clearly the Hamiltonian $H$ has the $1:2$ resonant symmetry, i.e., $H \circ \Psi^{1:2}_{\theta} = H$ for any $e^{\rmi\theta} \in \mathbb{S}^{1}$ (see \eqref{eq:Psi} for the definition of the action $\Psi$), and thus 
  \begin{equation*}
    R_{1:2}(\mathbf{a}) = \frac{1}{2}\parentheses{ \frac{|a_{1}|^{2}}{2} + |a_{2}|^{2} }
  \end{equation*}
  is conserved along the dynamics.
  On the other hand, the map $\mathbf{J}_{1:2}\colon \C_{\times}^{2} \to \su(2)^{*}$ takes the form
  \begin{equation*}
    \mathbf{J}_{1:2}(\mathbf{a}) = \parentheses{
      \Re\parentheses{ \frac{a_{1}^{2} \bar{a}_{2}}{\sqrt{2} |a_{1}|} },\,
      \Im\parentheses{ \frac{a_{1}^{2} \bar{a}_{2}}{\sqrt{2} |a_{1}|} },\,
      \frac{1}{2}\parentheses{ \frac{|a_{1}|^{2}}{2} - |a_{2}|^{2} }
    }.
  \end{equation*}
  
  Let us define the Hamiltonian $h\colon \su(2)^{*} \to \R$ by $h \circ \mathbf{J}_{1:2} = H$.
  This yields
  \begin{equation*}
    h(\mu) = 2\mu_{1}\sqrt{ \norm{\mu} + \mu_{3} },
  \end{equation*}
  where $\norm{\mu} = \sqrt{\mu_{1}^{2} + \mu_{2}^{2} + \mu_{3}^{2}}$.
  Then the Kummer shape is defined by $\norm{\mu} = r$ for the constant $r \defeq R_{1:2}(\mathbf{a}_{0})$ defined by the initial condition $\mathbf{a}_{0} \in \C_{\times}^{2}$ for the above dynamics.
  
  Now, Theorem~\ref{thm:dual_pair-n:m} implies that setting $\mu = (\mu_{1},\mu_{2},\mu_{3}) = \mathbf{J}_{1:2}(\mathbf{a}) \in \su(2)^{*}$ reduces the dynamics to a Lie--Poisson dynamics in $\su(2)^{*}$---more specifically on the coadjoint orbit or the Kummer shape $\norm{\mu} = c$---with respect to the Lie--Poisson bracket~\eqref{eq:LPB-su2} and the above Hamiltonian $h$.
  In fact, the Lie--Poisson equation~\eqref{eq:LPB-su2_coorinates} yields
  \begin{equation}
    \label{eq:reduced_1:2-dynamics}
    \dot{\mu}_{1} = -\frac{2 \mu_{1} \mu_{2}}{\sqrt{r + \mu_{3}}},
    \qquad
    \dot{\mu}_{2} = \frac{2 \mu_{1}^{2} - 4\mu_{3}(r + \mu_{3})}{\sqrt{r + \mu_{3}}},
    \qquad
    \dot{\mu}_{3} = 4\mu_{2}\sqrt{r + \mu_{3}}
  \end{equation}
  on the Kummer shape $\norm{\mu} = r$.
  
  The orbit of the above Lie--Poisson dynamics is given by the intersection of the sphere $\norm{\mu} = r$ and the level set of the Hamiltonian $h$; see Fig.~\ref{fig:1:2_resonance}.
  On the other hand, the standard Kummer shape in the 1:2 resonance would be a ``turnip''~\cite[Section~4.4.2]{Ho2011a}, i.e., one of the poles of the sphere is pinched, and the Poisson bracket in the reduced space $\su(2)^{*}$ is \textit{not} the standard Lie--Poisson bracket; see Holm and Vizman~\cite{HoVi2012}.
  \begin{figure}[htbp]
    \centering
    \includegraphics[width=.5\linewidth]{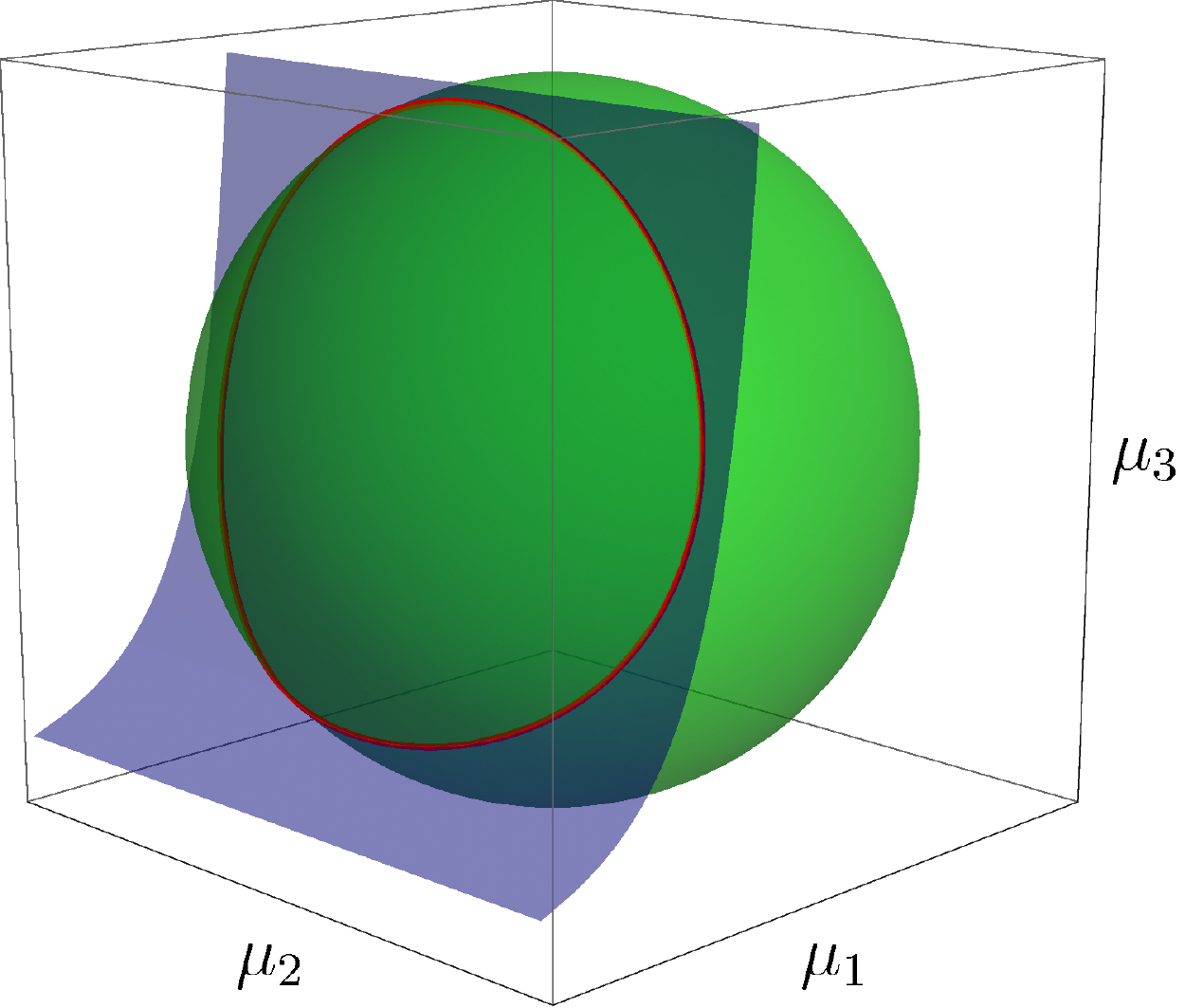}
    \caption{
      The Kummer shape is regularized to be the sphere (green), and the reduced dynamics (red)~\eqref{eq:reduced_1:2-dynamics} is at the intersection of the sphere and the level set (blue) of the Hamiltonian $h$.
   }
    \label{fig:1:2_resonance}
  \end{figure}
\end{example}

\subsection{{\boldmath $n:-m$} Resonances}
\label{sec:n:-m}
We may easily extend the above construction to those cases where one of the frequencies of resonance is negative.
Without loss of generality, let us consider $n:-m$ resonances with $n, m \in \N$.
So we consider the action
\begin{equation*}
  \Psi^{n:-m}_{(\cdot)}\colon \mathbb{S}^{1} \times \C_{\times}^{2} \to \C_{\times}^{2};
  \qquad
  (e^{\rmi\theta},(a_{1},a_{2})) \mapsto (e^{\rmi n\theta}a_{1}, e^{-\rmi m\theta}a_{2}).
\end{equation*}
on $\C_{\times}^{2}$ equipped with \eqref{eq:Omega-C2}.
However, equivalently, one may redefine $\bar{a}_{2}$ as $a_{2}$ and instead consider the action $\Psi^{n:m}$ given in \eqref{eq:Psi} on $\C_{\times}^{2}$ equipped with the symplectic form
\begin{equation*}
  \Omega_{\C_{\times}^{2}}^{1:-1} \defeq -\frac{\rmi}{2} \sum_{j=1}^{2} k_{j} \d b_{j} \wedge \d\bar{b}_{j}
  = -\d\Theta_{\C_{\times}^{2}}^{1:-1}
\end{equation*}
with $(k_{1},k_{2}) = (1,-1)$, where
\begin{equation*}
  \Theta_{\C_{\times}^{2}}^{1:-1} \defeq \frac{1}{2} \sum_{j=1}^{2} k_{j} \Im(\bar{b}_{j}\d b_{j}).
\end{equation*}

It is a straightforward computation as in $n:m$ resonances to check that $f_{n:m}$ is a local symplectomorphism with respect to $\Omega_{\C_{\times}^{2}}^{1:-1}$ as well as that $f_{n:m}$ is Poisson with respect to the corresponding Poisson bracket:
Defining
\begin{equation*}
  \PB{F}{G}_{\C_{\times}^{2}}^{1:-1} \defeq 2\rmi \sum_{j=1}^{2} k_{j} \parentheses{ \pd{F}{b_{j}} \pd{G}{\bar{b}_{j}} - \pd{G}{b_{j}} \pd{F}{\bar{b}_{j}} }
\end{equation*}
with $(k_{1},k_{2}) = (1,-1)$, we have
\begin{equation*}
  \PB{F\circ f_{n:m}}{G\circ f_{n:m}}_{\C_{\times}^{2}}^{1:-1} = \PB{F}{G}_{\C_{\times}^{2}}^{1:-1}\circ f_{n:m}.
\end{equation*}

We also define $R_{n:-m}\colon \C_{\times}^{2} \to \R$ as
\begin{equation*}
  R_{n:-m}(\mathbf{a}) \defeq \frac{1}{2}\parentheses{ \frac{|a_{1}|^{2}}{m} - \frac{|a_{2}|^{2}}{n} },
\end{equation*}
which satisfies $R_{n:-m} = R_{1:-1} \circ f_{n:m}$.

Let $K \defeq \diag(k_{1},k_{2}) =
\begin{tbmatrix}
  1 & 0 \\
  0 & -1
\end{tbmatrix}$ and 
\begin{align*}
  \SU(1,1)
  &\defeq
    \setdef{
    U \in \C^{2\times 2}
    }{ U^{*} K U = K,\, \det U = 1 } \\
  &= \setdef{
    \begin{bmatrix}
      \alpha & \beta \\
      \bar{\beta}  & \bar{\alpha}
    \end{bmatrix}
                     }{ \alpha, \beta \in \C,\, |\alpha|^{2} - |\beta|^{2} = 1 },
\end{align*}
and consider the natural action of $\SU(1,1)$ on $(\C_{\times}^{2}, \Omega_{\C_{\times}^{2}}^{1:-1})$.
Then the corresponding momentum map $\mathbf{J}_{1:-1}\colon \C_{\times}^{2} \to \su(1,1)^{*}$ is given by
\begin{align*}
  \mathbf{J}_{1:-1}(\mathbf{b})
  &= \rmi \parentheses{ K \mathbf{b}\mathbf{b}^{*} - \frac{1}{2}\tr(K \mathbf{b}\mathbf{b}^{*}) I } \\
  &= \parentheses{
    \Re(b_{1}\bar{b}_{2}),\,
    -\Im(b_{1}\bar{b}_{2}),\,
    \frac{1}{2} (|b_{1}|^{2} + |b_{2}|^{2})
  },
\end{align*}
It is clearly equivariant and thus $\mathbf{J}_{1:-1}$ is a Poisson map with respect to $\Omega_{\C_{\times}^{2}}^{1:-1}$ and the $(+)$-Lie--Poisson bracket on $\su(1,1)^{*}$.
We denote $\su(1,1)^{*}$ with the $(+)$-Lie--Poisson bracket by $\su(1,1)^{*}_{+}$ below.

As shown in Holm and Vizman~\cite[Theorem~8.1]{HoVi2012} (see also Iwai~\cite{Iw1985}), $R_{1:-1}$ and $\mathbf{J}_{1:-1}$ constitute a dual pair.
Hence so do $R_{n:-m}$ and $\mathbf{J}_{n:-m}$ as well, following the same argument as in $n:m$ resonance case.
The diagram below summarizes this result.
\begin{equation*}
  \begin{tikzcd}[column sep=9ex, row sep=9ex]
    & (\C_{\times}^{2}, \Omega_{\C_{\times}^{2}}^{1:-1}) \arrow{d}{f_{n:m}} \arrow[swap]{dl}{R_{n:-m}} \arrow{dr}{\mathbf{J}_{n:-m}}& \\
    \R & (\C_{\times}^{2}, \Omega_{\C_{\times}^{2}}^{1:-1}) \arrow{l}{R_{1:-1}} \arrow[swap]{r}{\mathbf{J}_{1:-1}} & \su(1,1)^{*}_{+}
  \end{tikzcd}
\end{equation*}

The Kummer shape in this case is a paraboloid for any $(n,m) \in \N^{2}$.
In fact, setting $\mu = \mathbf{J}_{n:-m}(\mathbf{a})$, we have
\begin{equation*}
  \mu_{3}^{2} - \mu_{1}^{2} - \mu_{2}^{2} = R_{n:-m}(\mathbf{a})^{2}.
\end{equation*}

\section{Generalization to Multi-dimensional Resonance}
\label{sec:MultiD_Resonance}
\subsection{Setup}
Let $\mathbf{a} = (a_{1}, \dots, a_{d})$ be coordinates for $\C_{\times}^{d}$, and generalize the symplectic form~\eqref{eq:Omega-C2} to $\C_{\times}^{d}$ as follows:
\begin{equation}
\label{eq:Omega-C^d}
  \Omega_{\C_{\times}^{d}} \defeq -\frac{\rmi}{2} \sum_{j=1}^{d} \d a_{j} \wedge \d\bar{a}_{j} = -\d\Theta_{\C_{\times}^{d}},
\end{equation}
where
\begin{equation*}
  \Theta_{\C_{\times}^{d}} \defeq \frac{1}{2} \sum_{j=1}^{d} \Im(\bar{a}_{j}\d a_{j}).
\end{equation*}
The associated Poisson bracket is 
\begin{equation}
  \label{eq:PB-C^d}
  \PB{F}{G}_{\C_{\times}^{d}} \defeq 2\rmi \sum_{j=1}^{d} \parentheses{ \pd{F}{a_{j}} \pd{G}{\bar{a}_{j}} - \pd{G}{a_{j}} \pd{F}{\bar{a}_{j}} }.
\end{equation}

We can also generalize the map $f_{n:m}$ introduced in \eqref{eq:f_n:m} earlier as follows:
\begin{proposition}
  \label{prop:f_n}
  Given a multi-index of natural numbers $\mathbf{n} \defeq (n_{1}, \dots, n_{d}) \in \N^{d}$, let us define $\{ \nu_{j} \}_{j \in \{1, \dots, d\}} \subset \N$ by
  \begin{equation*}
    \nu_{j} \defeq \prod_{\substack{1\le i \le d\\i\neq j}} n_{i},
  \end{equation*}
  and consider the map
  \begin{equation*}
    f_{\mathbf{n}}\colon \C_{\times}^{d} \to \C_{\times}^{d};
    \qquad
    \mathbf{a} \mapsto \parentheses{ \frac{a_{1}^{\nu_{1}}}{\sqrt{\nu_{1}}\,|a_{1}|^{\nu_{1}-1}}, \dots, \frac{a_{d}^{\nu_{d}}}{\sqrt{\nu_{d}}\,|a_{d}|^{\nu_{d}-1}} }.
  \end{equation*}
  Then $f_{\mathbf{n}}$ is a Poisson map as well as a local symplectomorphism.
\end{proposition}
\begin{proof}
  let $\mathbf{b} = (b_{1}, \dots, b_{d})$ be the coordinates for the second copy of $\C_{\times}^{d}$.
  Then the map $f_{\mathbf{n}}$ is written as $\mathbf{b} = f_{\mathbf{n}}(\mathbf{a})$, and one sees that, for any $j \in \{1, \dots, d\}$,
  \begin{equation*}
    \pd{b_{j}}{a_{j}}
    = \frac{\nu_{j} + 1}{2\sqrt{\nu_{j}}} \parentheses{ \frac{a_{j}}{|a_{j}|} }^{\nu_{j} - 1},
    \qquad
    \pd{b_{j}}{\bar{a}_{j}}
    = -\frac{\nu_{j} - 1}{2\sqrt{\nu_{j}}} \parentheses{ \frac{a_{j}}{|a_{j}|} }^{\nu_{j} + 1},
  \end{equation*}
  where the summation on $j$ is \textit{not} assumed.
  This implies that, for any $F, G \in C^{\infty}(\C_{\times}^{d})$,
  \begin{equation*}
    \pd{(F \circ f_{\mathbf{n}})}{a_{j}} \pd{(G \circ f_{\mathbf{n}})}{\bar{a}_{j}}
    - \pd{(G \circ f_{\mathbf{n}})}{a_{j}} \pd{(F \circ f_{\mathbf{n}})}{\bar{a}_{j}}
    = \pd{F}{b_{j}} \pd{G}{\bar{b}_{j}} - \pd{G}{b_{j}} \pd{F}{\bar{b}_{j}}
  \end{equation*}
  as well as
  \begin{equation*}
    \bar{b}_{j}\d b_{j} - b_{j}\d\bar{b}_{j} = \bar{a}_{j}\d a_{j} - a_{j}\d\bar{a}_{j}
    \iff
    \Im(\bar{b}_{j}\d b_{j}) = \Im(\bar{a}_{j}\d a_{j}).
  \end{equation*}
  The former equality implies
  \begin{equation*}
    \PB{(F \circ f_{\mathbf{n}})}{(G \circ f_{\mathbf{n}})}_{\C_{\times}^{d}} = \PB{F}{G}_{\C_{\times}^{d}} \circ f_{\mathbf{n}},
  \end{equation*}
  and hence $f_{\mathbf{n}}$ is Poisson, whereas the latter implies that $f_{\mathbf{n}}$---which is a local diffeomorphism although it is not globally one-to-one---locally leaves $\Theta_{\C_{\times}^{d}}$ invariant and hence $\Omega_{\C_{\times}^{d}}$ as well.
\end{proof}

\subsection{Momentum Maps}
Let us consider the $\mathbb{S}^{1}$-action
\begin{equation}
  \label{eq:Psi^n}
  \Psi^{\mathbf{n}}_{(\cdot)}\colon \mathbb{S}^{1} \times \C_{\times}^{d} \to \C_{\times}^{d};
  \qquad
  (e^{\rmi\theta},\mathbf{a}) \mapsto (e^{\rmi n_{1}\theta}a_{1}, \dots, e^{\rmi n_{d}\theta}a_{d}) \eqdef \Psi^{\mathbf{n}}_{\theta}(\mathbf{a}).
\end{equation}
It is clear that $\Psi^{\mathbf{n}}_{(\cdot)}$ leaves the canonical one-form $\Theta_{\C_{\times}^{d}}$ invariant, i.e., $(\Psi^{\mathbf{n}}_{\theta})^{*}\Theta_{\C_{\times}^{d}} = \Theta_{\C_{\times}^{d}}$ for any $e^{\rmi\theta} \in \mathbb{S}^{1}$, and hence is symplectic with respect to $\Omega_{\C_{\times}^{d}}$.
The corresponding momentum map is
\begin{equation*}
  \frac{1}{2} \sum_{j=1}^{d} n_{j} |a_{j}|^{2} = \mathcal{N}\,R_{\mathbf{n}}(\mathbf{a}),
\end{equation*}
where $\mathcal{N} \defeq \prod_{j=1}^{d} n_{j}$ and we defined $R_{\mathbf{n}}\colon \C_{\times}^{d} \to \R$ as
\begin{equation}
  \label{eq:R_n}
  R_{\mathbf{n}}(\mathbf{a}) \defeq \frac{1}{2} \sum_{j=1}^{d} \frac{|a_{j}|^{2}}{\nu_{j}}.
\end{equation}
Clearly we have $R_{\mathbf{n}} = R_{\mathbf{1}} \circ f_{\mathbf{n}}$.

Let us also consider a natural $\SU(d)$-action on $\C_{\times}^{d}$, i.e.,
\begin{equation}
  \label{eq:SU(d)-action}
  \Phi_{(\cdot)}\colon \SU(d) \times \C_{\times}^{d} \to \C_{\times}^{d},
  \qquad
  (U,\mathbf{b}) \mapsto U\mathbf{b}.
\end{equation}
and find an expression for the corresponding momentum map for the special case $\mathbf{n} = \mathbf{1} \defeq (1, \dots, 1) \in \N^{d}$:
\begin{lemma}
  \label{lem:J_1}
  The momentum map $\mathbf{J}_{\mathbf{1}}\colon \C_{\times}^{d} \to \su(d)^{*}$ corresponding to the above $\SU(d)$-action~\eqref{eq:SU(d)-action} is given by
  \begin{equation*}
    \mathbf{J}_{\mathbf{1}}(\mathbf{b}) = \rmi\parentheses{ \mathbf{b}\mathbf{b}^{*} - \frac{1}{d}|\mathbf{b}|^{2} I }.
  \end{equation*}
  It is a Poisson map with respect to $\Omega_{\C_{\times}^{d}}$ and the $(+)$-Lie--Poisson bracket on $\su(d)^{*}$.
\end{lemma}
\begin{proof}
  Let us first find the momentum map $\tilde{\mathbf{J}}\colon \C_{\times}^{d} \to \u(d)^{*}$ corresponding to the $\U(d)$-action defined the same manner as \eqref{eq:SU(d)-action}.
  Let $\xi \in \u(d)$ be arbitrary.
  Then the corresponding infinitesimal generator is given by $\xi_{\C_{\times}^{d}}(\mathbf{b}) = \xi\mathbf{b}$.
  Since this action clearly leaves $\Theta_{\C_{\times}^{d}}$ invariant, the momentum map $\tilde{\mathbf{J}}$ is defined by
  \begin{equation*}
    \ip{ \tilde{\mathbf{J}}(\mathbf{b}) }{\xi}
    = \ip{ \Theta_{\C_{\times}^{d}}(\mathbf{b}) }{ \xi\mathbf{b} },
  \end{equation*}
  where we define an inner product on $\u(d)$ as follows:
  \begin{equation*}
    \ip{\eta}{\xi} \defeq \frac{1}{2}\tr(\eta^{*}\xi).
  \end{equation*}
  We may then identify $\u(d)^{*}$ with $\u(d)$ and $\su(d)^{*}$ with $\su(d)$ via the above inner product.
  Now,
  \begin{align*}
    \ip{\Theta_{\C_{\times}^{d}}(\mathbf{b})}{\xi\mathbf{b}}
    &= \frac{1}{2}\Im(\mathbf{b}^{*}\xi\mathbf{b}) \\
    &= -\frac{\rmi}{2} \mathbf{b}^{*}\xi\mathbf{b} \\
    &= -\frac{\rmi}{2} \tr(\mathbf{b}\mathbf{b}^{*}\xi) \\
    &= \frac{1}{2} \tr\parentheses{ (\rmi\mathbf{b}\mathbf{b}^{*})^{*}\xi } \\
    &= \ip{ \rmi\mathbf{b}\mathbf{b}^{*} }{\xi},
  \end{align*}
  where we used the fact that $\xi^{*} = -\xi$ and hence $\mathbf{b}^{*}\xi\mathbf{b}$ is a pure imaginary number.
  So we have $\tilde{\mathbf{J}}(\mathbf{b}) = \rmi\,\mathbf{b}\mathbf{b}^{*}$.

  Now note that the action $\Phi$ in \eqref{eq:SU(d)-action} is the induced subgroup action of of the above $\U(d)$-action.
  Let $\iota\colon \su(d) \to \u(d)$ be the inclusion and $\iota^{*}\colon \u(d)^{*} \to \su(d)^{*}$ be its dual.
  Then the momentum map $\mathbf{J}_{\mathbf{1}}$ is given by $\mathbf{J}_{\mathbf{1}} = \iota^{*} \circ \tilde{\mathbf{J}}$; see, e.g., Marsden and Ratiu~\cite[Exercise~11.4.2]{MaRa1999}.

  By definition, the dual map $\iota^{*}\colon \u(d)^{*} \to \su(d)^{*}$ satisfies
  \begin{equation*}
    \ip{\iota^{*}(\mu)}{\xi}
    = \ip{\mu}{\iota(\xi)}
    = \ip{\mu|_{\su(d)}}{\xi},
  \end{equation*}
  and hence $\iota^{*}(\mu) = \mu|_{\su(d)}$.
  It is easy to see that the orthogonal complement of $\su(d)$ in $\u(d)$ in terms of the above inner product is given by
  \begin{equation*}
    \su(d)^{\perp} = \Span\braces{ \rmi\sqrt{\frac{2}{d}}\,I }.
  \end{equation*}
  Therefore, using the identification $\u(d)^{*} \cong \u(d)$ and $\su(d)^{*} \cong \su(d)$, the dual map $\iota^{*}$ is given by the orthogonal projection onto $\su(d)$:
  \begin{align*}
    \iota^{*}(\mu) &= \mu|_{\su(d)} \\
                   &= \mu - \ip{\rmi\sqrt{\frac{2}{d}}\,I}{\mu}\rmi\sqrt{\frac{2}{d}}\,I \\
                   &= \mu - \frac{1}{d}\tr(\mu)I.
  \end{align*}
  Therefore, we obtain
  \begin{equation*}
    \mathbf{J}_{\mathbf{1}}(\mathbf{b})
    = \iota^{*} \circ \tilde{\mathbf{J}}(\mathbf{b})
    = \rmi\parentheses{ \mathbf{b}\mathbf{b}^{*} - \frac{1}{d}|\mathbf{b}|^{2} I }. \qedhere
  \end{equation*}
\end{proof}

\subsection{Dual Pairs}
Now we are ready to generalize Theorem~\ref{thm:dual_pair-n:m} to the above multi-dimensional setting.
Let $\su(d)^{*}_{+}$ denote $\su(d)^{*}$ equipped with the $(+)$-Lie--Poisson bracket on $\su(d)^{*}$, and define $\mathbf{J}_{\mathbf{n}}\colon \C_{\times}^{d} \to \su(d)^{*}_{+}$ as $\mathbf{J}_{\mathbf{n}} \defeq \mathbf{J}_{\mathbf{1}} \circ f_{\mathbf{n}}$.
Then we have the following generalization:
\begin{theorem}
  \label{thm:dual_pair-n}
  The Poisson maps $R_{\mathbf{n}}\colon \C_{\times}^{d} \to \R$ and $\mathbf{J}_{\mathbf{n}}\colon \C_{\times}^{d} \to \mathbf{J}_{\mathbf{1}}(\C_{\times}^{d}) \subset \su(d)^{*}_{+}$ are a dual pair for any multi-index $\mathbf{n} \in \N^{d}$ of $d$ natural numbers, i.e., for any $\mathbf{a} \in \C_{\times}^{d}$, $\ker T_{\mathbf{a}}R_{\mathbf{n}}$ and $\ker T_{\mathbf{a}}\mathbf{J}_{\mathbf{n}}$ are symplectic orthogonal complements to each other.
  Moreover, the dual pair of Poisson maps for the $\mathbf{n}$ resonances is related to the dual pair of momentum maps $R_{\mathbf{1}}$ and $\mathbf{J}_{\mathbf{1}}$ for the $\mathbf{1}$-resonance as is shown in the diagram below.
  \begin{equation*}
    \begin{tikzcd}[column sep=9ex, row sep=9ex]
      & (\C_{\times}^{d}, \Omega_{\C_{\times}^{d}}) \arrow{d}{f_{\mathbf{n}}} \arrow[swap]{dl}{R_{\mathbf{n}}} \arrow{dr}{\mathbf{J}_{\mathbf{n}}}& \\
      \R & (\C_{\times}^{d}, \Omega_{\C_{\times}^{d}}) \arrow{l}{R_{\mathbf{1}}} \arrow[swap]{r}{\mathbf{J}_{\mathbf{1}}} & \mathbf{J}_{\mathbf{1}}(\C_{\times}^{d}) \subset \su(d)^{*}_{+}
    \end{tikzcd}
  \end{equation*}
\end{theorem}

\begin{proof}
  First consider the special case with $\mathbf{n} = \mathbf{1}$.
  We note in passing that this case is also treated in Cari{\~n}ena et~al.~\cite[Section~5.4.5.3]{CaIbMaMo2014}.
  It is clear from \eqref{eq:R_n} that $R_{\mathbf{1}}$ is $\mathbb{S}^{1}$ invariant as well as $\SU(d)$ invariant, whereas $\mathbf{J}_{\mathbf{1}}$ is equivariant: From Lemma~\ref{lem:J_1}, for any $U \in \SU(d)$, we have
  \begin{equation*}
    \mathbf{J}_{\mathbf{1}}( \Phi_{U}(\mathbf{b}) )
    = \Ad_{U^{-1}}^{*}\mathbf{J}_{\mathbf{1}}(\mathbf{b}).
  \end{equation*}
  Therefore, both $R_{\mathbf{1}}$ and $\mathbf{J}_{\mathbf{1}}$ are Poisson maps; particularly the latter is Poisson with respect to the canonical Poisson bracket~\eqref{eq:PB-C^d} on $\C_{\times}^{d}$ and the $(+)$-Lie--Poisson bracket on $\su(d)^{*}$.
  
  One also sees that $\SU(d)$ acts on the level sets of $R_{\mathbf{1}}$ transitively via the above action $\Phi$ as follows:
  The level set $R_{\mathbf{1}}^{-1}(r)$ of $R_{\mathbf{1}}$ with any $r > 0$ is a $(2d-1)$-dimensional sphere in $\C_{\times}^{d}$ (those points corresponding to the removed origins of the copies of $\C_{\times}$ are removed) centered at the (removed) origin, and thus $\SU(d)$ acts on each level set transitively.
  It is also clear that every point in $\C_{\times}^{d}$ is a regular point of $R_{\mathbf{1}}$ and $\mathbf{J}_{\mathbf{1}}$; notice that the codomain of $\mathbf{J}_{\mathbf{1}}$ is restricted to the image $\mathbf{J}_{\mathbf{1}}(\C_{\times}^{d})$ in $\su(d)^{*}_{+}$.
  Therefore, by Theorem~2.1 of \cite{HoVi2012}, $R_{\mathbf{1}}$ and $\mathbf{J}_{\mathbf{1}}$ constitute a dual pair.
  
  The extension to an arbitrary $\mathbf{n} \in \N^{d}$ is a simple generalization of the proof of Theorem~\ref{thm:dual_pair-n:m} using Proposition~\ref{prop:f_n} as the above diagram shows: Note that we have $R_{\mathbf{n}} = R_{\mathbf{1}} \circ f_{\mathbf{n}}$ and $\mathbf{J}_{\mathbf{n}} = \mathbf{J}_{\mathbf{1}} \circ f_{\mathbf{n}}$ here.
\end{proof}

\begin{example}[1:1:2 resonance]
  \label{ex:112}
  Let $d = 3$ and consider the dynamics in $\C_{\times}^{3}$ with respect to the symplectic structure~\eqref{eq:Omega-C^d} and the Hamiltonian
  \begin{equation*}
    H(\mathbf{a}) = \Re\parentheses{ a_{1}^{2}(\bar{a}_{2}^{2} + \bar{a}_{3}) }.
  \end{equation*}
  The Hamiltonian system $\ins{X_{H}}\Omega_{\C_{\times}^{3}} = \d{H}$ yields
  \begin{equation*}
    \dot{a}_{1} = 2\rmi\,\bar{a}_{1} (a_{2}^{2} + a_{3}),
    \qquad
    \dot{a}_{2} = 2\rmi\,a_{1}^{2} \bar{a}_{2},
    \qquad
    \dot{a}_{3} = \rmi\,a_{1}^{2}.
  \end{equation*}
  The Hamiltonian $H$ has $1:1:2$ resonant symmetry, i.e., $H \circ \Psi^{\mathbf{n}}_{\theta} = H$ with $\mathbf{n} = (1,1,2)$ for any $e^{\rmi\theta} \in \mathbb{S}^{1}$ (see \eqref{eq:Psi^n} for the definition of the action $\Psi$), and thus 
  \begin{equation*}
    R_{\mathbf{n}}(\mathbf{a}) = \frac{1}{4}\parentheses{ |a_{1}|^{2} + |a_{2}|^{2} + 2|a_{3}|^{2} }
  \end{equation*}
  is a conserved quantity for the dynamics.

  Let us use a variant $\{ \gamma_{j} \}_{j=1}^{8} \subset \su(3)$ of the Gell-Mann matrices~\cite{Ge1962} as a basis for $\su(3)$ to identify $\su(3)$ with $\R^{8}$:
  For any $\xi \in \su(3)$,
  \begin{equation}
    \label{eq:su3-basis}
    \xi = \sum_{j=1}^{8} \xi_{j} \gamma_{j}
    = \rmi
    \begin{bmatrix}
      \xi_{3}+{\xi_{8}}/{\sqrt{3}} & \xi_{1}+\rmi \xi_{2} & \xi_{4}+\rmi \xi_{5} \smallskip\\
      \xi_{1}-\rmi \xi_{2} & {\xi_{8}}/{\sqrt{3}}-\xi_{3} & \xi_{6}+\rmi \xi_{7} \smallskip\\
      \xi_{4}-\rmi \xi_{5} & \xi_{6}-\rmi \xi_{7} & -{2 \xi_{8}}/{\sqrt{3}}
    \end{bmatrix}
    \mapsto \boldsymbol{\xi} = (\xi_{1}, \dots, \xi_{8}) \in \R^{8}.
  \end{equation}
  We also identify $\su(3)^{*}$ with $\su(3)$ as well just as described in the proof of Lemma~\ref{lem:J_1}.
  
  The map $f_{\mathbf{n}}\colon \C_{\times}^{3} \to \C_{\times}^{3}$ is defined as
  \begin{equation*}
    f_{\mathbf{n}}(\mathbf{a}) \defeq \parentheses{ \frac{a_{1}^{2}}{\sqrt{2}\,|a_{1}|},\, \frac{a_{2}^{2}}{\sqrt{2}\,|a_{2}|},\, a_{3} },
  \end{equation*}
  and $\mathbf{J}_{\mathbf{n}}\colon \C_{\times}^{3} \to \su(3)^{*}$ takes the form
  \begin{multline*}
    \mathbf{J}_{\mathbf{n}}(\mathbf{a}) = \Biggl(
    \Re\parentheses{ \frac{a_{1}^{2} \bar{a}_{2}^{2}}{2 |a_{1}| |a_{2}|} },\,
    \Im\parentheses{ \frac{a_{1}^{2} \bar{a}_{2}^{2}}{2 |a_{1}| |a_{2}|} },\,
    \frac{1}{4}\parentheses{ |a_{1}|^{2} - |a_{2}|^{2} },\, \\
    \Re\parentheses{ \frac{a_{1}^{2} \bar{a}_{3}}{\sqrt{2} |a_{1}|} },\,
    \Im\parentheses{ \frac{a_{1}^{2} \bar{a}_{3}}{\sqrt{2} |a_{1}|} },\,
    \Re\parentheses{ \frac{a_{2}^{2} \bar{a}_{3}}{\sqrt{2} |a_{2}|} },\,
    \Im\parentheses{ \frac{a_{2}^{2} \bar{a}_{3}}{\sqrt{2} |a_{2}|} },\, \\
    \frac{1}{4\sqrt{3}}\parentheses{ |a_{1}|^{2} + |a_{2}|^{2} - 4 |a_{3}|^{2} }
    \Biggr).
  \end{multline*}

  We define the reduced Hamiltonian
  \begin{equation*}
    h(\mu) \defeq 4 \mu_{1} \sqrt{\mu_{1}^2+\mu_{2}^2} + 2 \mu_{4} \parentheses{ \frac{(\mu_{1}^{2} + \mu_{2}^{2})(\mu_{4}^{2} + \mu_{5}^{2})}{\mu_{6}^{2} + \mu_{7}^{2}} }^{1/4}
  \end{equation*}
  on the open subset
  \begin{equation*}
    \setdef{ \mu \in \su(3)^{*} }{ (\mu_{1},\mu_{2}) \neq 0,\, (\mu_{4},\mu_{5}) \neq 0,\, (\mu_{6},\mu_{7}) \neq 0 }
  \end{equation*}
  so that it satisfies $h \circ \mathbf{J}_{\mathbf{n}} = H$.
  The reduced dynamics is then given by the Lie--Poisson equation
  \begin{equation*}
    \dot{\mu} = -\ad_{\tpd{h}{\mu}}^{*}\mu.
  \end{equation*}

  One advantage of our formulation is that one can find the Casimirs relatively easily because the Lie--Poisson bracket is standard.
  In fact, it is well known that $\su(3)^{*}$ has quadratic and cubic Casimirs:
  \begin{equation*}
    C_{2}(\mu) \defeq \sum_{j=1}^{8}\mu_{j}^{2},
    \qquad
    C_{3}(\mu) \defeq \sum_{1\le j,k,l\le 8} d_{jkl}\mu_{j} \mu_{k} \mu_{l},
  \end{equation*}
  where the coefficients $\{ d_{jkl} \}_{1\le j,k,l\le 8}$ are defined so that the basis $\{ \gamma_{j} \}_{j=1}^{8}$ for $\su(3)$ defined in \eqref{eq:su3-basis} satisfies, for any $j, k \in \{1, \dots, 8\}$,
  \begin{equation*}
    \gamma_{j}\gamma_{k} + \gamma_{k}\gamma_{j} = -\frac{4}{3}\delta_{jk}I + 2\rmi \sum_{l=1}^{8} d_{jkl} \gamma_{l}.
  \end{equation*}
  This results in the following non-zero coefficients (all the others vanish):
  \begin{gather*}
    d_{118} = d_{228} = d_{338} = -d_{888} = \frac{1}{\sqrt{3}}, \\
    d_{146} = d_{157} = -d_{247} = d_{256} = d_{344} = d_{355} = -d_{366} = -d_{377} = \frac{1}{2}, \\
    d_{448} = d_{558} = d_{668} = d_{778} = -\frac{1}{2\sqrt{3}}.
  \end{gather*}
  These two Casimirs are conserved along the Lie--Poisson dynamics.
\end{example}

While the geometry of the multi-dimensional generalization of the dual pairs works out nicely, it is not clear if this dual pair is particularly effective in understanding multi-dimensional dynamics in resonance.
In fact, in the above example, the resulting Lie--Poisson equation is defined in a higher-dimensional space, $\su(3)^{*} \cong \R^{8}$, than the original one, $\C_{\times}^{3}$.
The extra conserved quantities $C_{2}$ and $C_{3}$ compensate for this increase in dimension, but unfortunately it is not evident whether the Lie--Poisson formulation has a clear advantage over the original formulation.

\section*{Acknowledgments}
I am grateful to the referees for their comments and criticisms, particularly for pointing out some subtleties of the results in Theorem~\ref{thm:dual_pair-n} and Example~\ref{ex:112}.
I would also like to thank Darryl Holm for his comments and suggestions on an earlier draft of the paper.
Happy 70th birthday, Darryl!

\medskip
Received xxxx 20xx; revised xxxx 20xx.
\medskip
\end{document}